\documentclass[reqno]{amsart}

\usepackage{amssymb}
\usepackage{amsmath}
\usepackage{amsthm}

\newtheorem{thm}{Theorem}
\newtheorem{lem}{Lemma}
\newtheorem{cor}{Corollary}
\usepackage{relsize}

\newcommand{\CP}{\mathbb{CP}^1}
\renewcommand{\P}{\mathbb{P}}
\newcommand{\E}{\mathbb{E}}

\newcommand{\kahler}{K\"ahler }

\newcommand{\abs}[1]{\left\vert#1\right\vert}

\newcommand{\su} {$\operatorname{SU(2)}$ }

\newcommand{\cov}{\operatorname{Cov}}
\newcommand{\hn}{H^0(M,L^n)}

\newcommand{\beq} {\begin{equation} }
	\newcommand{\eeq} {\end{equation} }

\renewcommand{\hat} {\widehat}
\renewcommand{\tilde}{\widetilde}

\begin{document}
	
 \title[Smallest distances]{Smallest distances between zeros of Gaussian analytic functions}
	
	\author[Feng]{Renjie Feng}
	
	\author[Yao]{Dong Yao}

	\address{Sydney Mathematical Research Institute, The University of Sydney, Australia.}
	\email{renjie.feng@sydney.edu.au}

	
	
	\address{School of Mathematics and Statistics/RIMS, Jiangsu Normal University, China.}
	\email{dongyao@jsnu.edu.cn}
 
	\date{\today}
	\maketitle

		
		
		
		\begin{abstract}
			
			In this article, we study the smallest distances between the zeros of  Gaussian analytic functions  over   compact Riemann surfaces.  Our main result is that, after appropriate rescaling, the point process of the smallest distances converge to a Poisson point process with a universal rate. Furthermore, the locations where these smallest distances occur tend to follow  a uniform measure with respect to the volume form.  As a consequence, the limiting density of the  $k$-th rescaled smallest distance is  proportional to $x^{4k-1}e^{-x^4}$ for any $k\geq 1$. Analogous results hold for the classical Gaussian Entire Functions. 
		\end{abstract}
		

	

	\section{Introduction}
	

Let $M$ be a smooth compact Riemann surface polarized by a positive holomorphic line bundle $L$.
We give $L$ a Hermitian metric $h$ with positive curvature $\Theta_h$ and give $M$ the \kahler form $\omega$ such that $\omega/\pi=\Theta_h$.   
Without loss of generality, we assume that the line bundle $L$ has degree one.  
  Since $\Theta_h$ is a de Rham representative of the first Chern class $c_1(L)$ (see, e.g., \cite[Chapter 1]{GH}), the degree of $L$ is given by the integral of $\Theta_h$. It follows that $$\int_M \frac{\omega}{\pi} = \int_M \Theta_h = \deg(L) = 1,$$
allowing \(\omega/\pi\) to be interpreted as a probability measure on \(M\). 	

Given two points $z$ and $w$ on the Riemann surface $(M,\omega)$,  we denote $d_{g}(z, w)$   as the geodesic distance  between them with respect to the Hermitian metric induced by the \kahler form $\omega$. 
 
	

	Given $s$,  a Gaussian random holomorphic section of the line bundle $L^n$ (as described in \S\ref{Gaussianva}), we define $Z_n(s)$ as  its random zero set. 
	The cardinality satisfies $|Z_n(s)|=n$ since the degree of $L^n$ is $n$. Dropping the dependence on $s$ for simplicity, we  write 
	\begin{equation}\label{defzn}
		Z_n=Z_n(s)=\{z_1, z_2, ...., z_n\}.
	\end{equation}
	
	Globally, the limiting distribution of  zeros is  given by the  probability measure $\omega/\pi$ (see \eqref{zerodis} below).   
	Locally, for a given zero, the universal rescaling limit of the correlation functions (see \eqref{universalk}) implies that, on average, there is another zero within a geodesic distance of order $n^{-1/2}$. In other words, the nearest neighbor distance between zeros is, on average, of order $n^{-1/2}$.
	
	In this article, we investigate a different local statistic, namely the smallest distances among all zeros. For random zeros in $Z_n$, there are $(n-1)n/2$ geodesic distances between them, and these distances are almost surely distinct. We can arrange them in ascending order as $\sigma_1<\sigma_2<\cdots <\sigma_{{(n-1)n}/2}$. The purpose of this article is to study the decay order of the $k$-th smallest distance $\sigma_k$ and its limiting distribution as $n\to\infty$ for any given $k\geq 1$.  
	%
	
	For each pair of distinct zeros $z_i$ and $z_j$ in $Z_n$, we choose $\hat {z}_{ij}$ uniformly from the set $\{z_i,z_j\}$, i.e., each point is selected with a probability of 1/2, to mark the location of the pair. We consider the following point process of the rescaled smallest distances together with their locations, 
	\beq\label{alldistanorm} \mathcal S_n =\sum_{1\leq i<j\leq n }\delta_{n^{3/4} d_g(z_i,z_j),\, \hat {z}_{ij}}. \eeq
	Our main result is  
	


	\begin{thm} \label{main1}   As $n\to\infty$,  $\mathcal S_n$  converges in distribution to a Poisson point process $\mathcal S$ on $[0,\infty)\times M$ with a universal rate  		$$\label{rate}\mathbb E [\mathcal S([0,a]\times U)]=\Big(\frac 18  a^4\Big) 
		\int_{U}  \frac{\omega}{\pi}, $$
		for any $a\in (0,\infty)$ and the measurable subset $U\subseteq M$. 
	\end{thm}
	As a direct corollary, we have 
	\begin{cor} \label{cdsds}Let $\sigma_k$ be the $k$-th smallest distance between all pairs of zeros,  
		then its limiting density satisfies 
		$$\lim_{n\to\infty} \mathbb P\Big(\big(\frac 18\big)^{1/4}n^{3/4}\sigma_k\in [x,x+ dx]\Big)=\frac{ 4x^{4k-1}}{(k-1)!}e^{-  x^4}dx.$$
		Furthermore, let $\pi_k$ be the location of $\hat z_{ij}$ such that $d_g(z_i,z_j)$ equals $\sigma_k$, then $\pi_k$ is asymptotically independent of $\sigma_k$, and tends to the uniform measure with respect to the volume form, i.e.,  for any $U\subseteq M$,
		$$
		\lim_{n\to\infty}\P(\pi_k\in U)=\int_{U}\frac\omega\pi. 
		$$
	\end{cor}

	The results indicate that the smallest distances of random zeros exhibit a decay of order  $n^{-{3}/{4}}$. These findings are universal and apply specifically to classical $\operatorname{SU(2)}$ random polynomials on the complex projective space $\CP\cong S^2$ and Gaussian random Riemann theta functions on the 1-dimensional complex torus  (see \S\ref{badc} for the precise definitions). 

	For the classical Gaussian Entire Functions (GEF), one can  derive the smallest distances for its zeros similarly.  The   GEF is given  by 
	$$\label{analytic12345}g(z)=\sum_{j\geq 0} \frac {b_j}{\sqrt{j!}}z^j,$$
	where $b_j$ are independent and identically distributed (i.i.d.) standard complex Gaussian random variables. 
	The distribution of the zeros of GEF is invariant under translations of $\mathbb{C}$, and it is given by $d\ell/\pi$, where $d\ell$ denotes the Lebesgue measure on $\mathbb C$.   Now we define $Z_R$ as the set of zeros of GEF falling in the disk $B_R(0)$ centered at the origin with radius $R$. We also define the point process of the smallest distances $\mathcal S_R$, in the same manner as $\mathcal S_n$ earlier, 
	$$ \mathcal S_R =\sum_{z_ i\neq z_j\in Z_R }\delta_{R^{1/2} d(z_i,z_j),\, \hat {z}_{ij}/R},$$
	where $d(\cdot, \cdot )$ is the Euclidean distance.  Analogously to $\mathcal S_n$, we have 
	
	\begin{thm} \label{main13}   As $R\to\infty$,  $\mathcal S_R$  converges in distribution to a Poisson point process $\mathcal S$ on $[0,\infty)\times B_1(0)$ with the  rate  		$$\label{rate}\mathbb E [\mathcal S([0,a]\times U)]=\Big(\frac 18  a^4\Big) 
		\int_{U}  \frac{d\ell}{\pi}, $$
		for any $a\in (0,\infty)$ and the measurable subset $U\subseteq B_1(0)$, where $d\ell$ is the Lebesgue measure on $\mathbb C$.
	\end{thm}
	Again, Theorem \ref{main13}  implies the limiting density of the $k$-th smallest distance between zeros of GEF as in Corollary \ref{cdsds}, and the locations where the smallest distances occur follow the uniform distribution. The proof of Theorem \ref{main13} follows the same arguments as the ones in Theorem \ref{main1}, and we omit its proof. 
	
	In comparison,  if one samples $n$ points  independently from a Binomial distribution on the unit sphere $S^2$ with respect to the uniform measure, the average value of the nearest neighbor distances is still of the order $n^{-1/2}$. However, the smallest distance between these Binomial points is of the order $n^{-1}$, which has a  smaller decay order than the random zeros (e.g., \cite{BRS}). The reason for this is that nearby Binomial points are neutral, whereas random zeros on the 1-dimensional Riemann surfaces discussed in this article repel each other (see \eqref{rep} below).

In a related but different setting, the smallest and largest gaps between successive zeros of smooth stationary Gaussian processes with sufficiently fast decaying covariance kernels were recently investigated in \cite{FGY} and \cite{FM}, respectively. 
In both works, it was shown that the point processes formed by suitably rescaled extreme gaps converge to Poisson point processes, from which one can determine the corresponding orders of the extreme gaps.

	

	So far we only discuss the cases of Riemann surfaces. For general complex $m$-dimensional compact \kahler manifolds, Bleher-Shiffman-Zelditch considered the intersection of zero locus of $1\leq \ell \leq m$ Gaussian random holomorphic sections, and they proved that the point correlation functions of  zeros still have some universal limit (\cite[Theorem 3.6]{BSZ1} and  \cite[Theorem 1.1]{BSZ4}).  For example, by choosing the normal coordinate around a point $z_0\in M$, the 2-point correlation function admits a pointwise universal limit $$\rho^\infty_{2,\ell, m}(u,v)=\lim_{n\to\infty}\rho_2(z_0+\frac u{\sqrt{n}}, z_0+\frac v{\sqrt{n}})/n^{2\ell}.$$ In particular,   when $\ell=m$, the zeros locus are discrete, and one has (\cite[Theorem 1.3]{BSZ4})
	\beq\label{mdim}\rho^\infty_{2,m, m}(u,v)= \frac{m+1}4|u-v|^{4-2m}+O(|u-v|^{8-2m})\,\,\, \mbox{as}\,\,\, u\to v. \eeq
	This implies that the discrete zeros are repulsive for $ m=1$, neutral for $m=2$ and attractive for $m\geq 3$. This indicates that the behavior of
	the smallest distances of the nearest neighbor zeros may be quite different depending on dimensions. 

	
	
	
	
	\bigskip
	
	\emph{Notation.} The symbols $c$ and $C$ represent positive constants independent of $n$, but their specific values may vary from line to line.  
	

	\section{Gaussian random holomorphic fields}\label{badc}
	In this section, we review  some basic concepts in complex geometry and introduce Gaussian random holomorphic sections over \kahler manifolds of arbitrary  dimension. However, we only focus  on 1-dimensional Riemann surfaces later on. For a general introduction to complex geometry, see \cite[Chapters~0 and~1]{GH}.
	\subsection{\kahler geometry }\label{randomsections}
	Let $(M,\omega)$ be a smooth compact K\"ahler manifold of complex dimension $m$.
Locally, on a coordinate chart $U\subset M$, the K\"ahler form can be written as
\[
\omega = \frac{\sqrt{-1}}{2}\partial\bar\partial \phi,
\]
where $\phi$ is a real-valued smooth K\"ahler potential.

Let $(L,h)\to M$ be a positive holomorphic line bundle equipped with a smooth
Hermitian metric $h$.   Its Chern curvature form is given  by
\[
\Theta_h = -\frac{\sqrt{-1}}{2\pi}\partial\bar\partial \log h(e,e),
\]
where $e$ is any local non-vanishing holomorphic section of $L$ over $U$.

We assume that $(L,h)$ polarizes $(M,\omega)$, namely that
\[
\Theta_h = \frac{\omega}{\pi}.
\]
Therefore,  one may choose a local
holomorphic frame $e$ of $L$ such that
\[
h(e,e) = e^{-\phi}:=|e|_h^2.
\]

	Finally we equip $M$ with the Hermitian metric $g$ corresponding to the \kahler form $\omega$ such that the induced Riemannian volume form is
	$$\label{volumeform}dV:= \frac{\omega^m}{m!}. $$
	
	We denote by $H^0(M,L^n)$ the space of global holomorphic sections of the $n$-th tensor power of $L$, and let $d_n:=\operatorname{dim} H^0(M,L^n)$. 	
	The Hermitian
	metric $h$ induces a Hermitian metric $h
	^n$ on $L^n$ as $|e^{ \otimes n}|_{h^n}=|e|_h^n$. 
	The Hermitian metric $h^n$ on $L^n$, together with the volume form $dV$ on $M$,
induces an inner product on the space of square-integrable sections $\mathcal L^2(M,L^n)$,
defined by
\begin{equation}\label{innera}
\langle s_1, s_2 \rangle_{h^n}
:= \int_M h^n(s_1,s_2)\, dV,
\qquad s_1, s_2 \in \mathcal L^2(M,L^n).
\end{equation}
We denote by
\[
\Pi_n : \mathcal L^2(M,L^n) \rightarrow H^0(M,L^n)
\]
the orthogonal projection. The integral kernel of the projection operator $\Pi_n$ is called the Bergman kernel. Let $\big\{s_{1},..., s_{d_n}\big\}$
	be an orthonormal basis for $\hn$ with respect to the inner product~\eqref{innera}, then the Bergman kernel reads, 
\[
\Pi_n (z,w)=	\sum_{j=1}^{d_n} s_j(z) \otimes s_j(w)^*.
\]
	

	A classical example  is the Fubini-Study metric of the hyperplane line bundle over the complex projective space $(\mathcal O(1), h_{FS})\to (\CP,\omega_{FS})$.  
	In an affine coordinate, the \kahler form and the \kahler potential for the Fubini-Study metric are
	$$
	\omega_{FS}=\frac{\sqrt{-1}}{2}\frac{dz\wedge d \bar z}{(1+|z|^2)^2},\,\,\,\phi_{FS}(z)=\log (1+|z|^2).
	$$ We equip
	$\mathcal O(1)$ with its
	Fubini-Study metric, where  we can choose a non-vanishing local frame $e(z)$ such that
	$$|e(z)|^2_{h_{FS}}=e^{-\phi_{FS}}=\frac 1{1+|z|^2}.$$
	An orthonormal basis of $H^0(\CP,\mathcal O(n))$ under the inner product \eqref{innera} is now given by
	\begin{equation}\label{basis1}
		\left\{\sqrt{\pi^{-1}(n+1) {n\choose j}}z^j \right\}_{j=0}^{n}.
	\end{equation}
	
	Another classical example is  the Riemann theta functions over Abelian varieties. We refer to \cite[p.~1--11]{Mu} and \cite[Section 6, Chapter 2]{GH} for more details. For simplicity, consider 1-dimensional complex torus $\mathbb T=\mathbb{C} /\mathbb{Z}+\sqrt{-1} \mathbb{Z}$. Write $z=x+\sqrt{-1}y$, where $x$ and $y \in \mathbb{R}$ can be viewed as the periodic coordinate of $\mathbb T$ with period 1. Consider  the Riemann theta function $$\theta(z)=\sum_{k\in\mathbb{Z}} e^{-\pi k^{2}+2\pi i k z}.$$ The zero divisor of $\theta$ on $\mathbb{T}$ consists of a single point  \(
p=\frac{1+\sqrt{-1}}{2}\quad (\mathrm{mod}\ \mathbb Z+i\mathbb Z)
\).  Define the line bundle $L:=\mathcal{O}(p)\to\mathbb{T}$.
By Riemann-Roch Theorem, $\dim H^{0}(\mathbb{T},L^n)=n$. Moreover, $H^{0}(\mathbb{T},L)$ is spanned by $\theta(z)$; and $H^0\left(\mathbb T, L^n\right)$ is generated by:
	$$
	\theta_j(z)=\sum_{k \in \mathbb{Z}} e^{-\pi n (\frac j n+k)^2+2 \pi \sqrt{-1}n(\frac j n+k)z}, \quad j \in \mathbb{Z} / n \mathbb{Z}.
	$$
	In addition, $\theta_j(z)$'s are holomorphic functions over $\mathbb{C}$ and satisfy the following quasi-periodicity relations:
	$$
	\theta_j(z+1)=\theta_j(z), \quad \theta_j(z+\sqrt{-1})=e^{-2 \pi \sqrt{-1}nz+n \pi} \theta_j(z).
	$$
	The set of functions $$\label{theta}\left\{{\pi^{-1/2}(2n)}^{1/4}\theta_j(z)\right\}_{j=0}^{n-1}$$ forms an orthonormal basis with respect to the inner product \eqref{innera}, under the \kahler form and the \kahler potential, $$\omega_{\mathbb T}=\frac{\pi\sqrt{-1}}{2} {dz\wedge d \bar z},\quad  \phi_{\mathbb T}=2\pi y^2.$$ 
	
	Finally, we introduce the \kahler normal coordinate. 
	Consider a complex $m$-dimensional \kahler manifold $(M, \omega)$, and let $z_0$ be a fixed point in $M$. We can choose a \kahler normal coordinate  $\big\{z^1,.., z^m\big\}$  in a neighborhood $U\subset M$ around $z_0$ such that $z_0=0$, and 
	the \kahler form at $z_0$ is  (\cite[equation~(1.3)]{T}): $$\label{adapted}\,\, \omega(z_0)= \frac {\sqrt{-1}} {2}\sum_{j=1}^m dz^j\wedge d\bar{z}^j,$$
and it has the following Taylor expansion around $z_0$ (\cite[p.~107]{GH}), 
	\begin{equation}\label{taylorforkahler} \omega(z)= \frac {\sqrt{-1}} {2}\sum \big(\delta_{ij}+[2]\big) dz^i\wedge d\bar{z}^j,
	\end{equation}
	where the $[2]$ here represents a term that vanishes up to order 2 at $z_0=0$. The same expansion holds for the Hermitian metric $g$ induced by $\omega$.

	\subsection{ Gaussian random holomorphic fields}\label{Gaussianva}

	The inner product \eqref{innera} induces a complex Gaussian probability measure on the space of global holomorphic sections $\hn$. Such Gaussian random holomorphic section has the expression: 
	\begin{equation}\label{standardgauss}
		s:=\sum_{j=1}^{d_n}b_j s_{j},\end{equation}
	where $\big\{s_{1},..., s_{d_n}\big\}$
	is an orthonormal basis for $\hn$ with respect to the inner product \eqref{innera}, and $\big\{b_1,..., b_{d_n}\big\}$  are i.i.d. standard complex Gaussian random
	variables that satisfy, 
	\beq\label{gaussiancomplex}\mathbb Eb_i=0,\quad \mathbb Eb_ib_j=0,\quad \mathbb Eb_i\bar b_j=\delta_{ij}.\eeq
	The covariance kernel of the Gaussian random holomorphic sections is given by the Bergman kernel
$$\begin{aligned}
\mathrm{Cov}(s(z), s(w)) 
&= \mathbb{E} \big( s(z) \otimes s(w)^* \big) = \sum_{j=1}^{d_n} s_j(z) \otimes s_j(w)^*=:\Pi_n(z, w). \\
\end{aligned}$$
	Under the local coordinate,  we write  $s_j=f_je^{\otimes n}$, then locally 
\begin{equation}\label{locald}
s(z)=\Big[ \sum_{j=1}^{d_n}b_j f_j(z)  \Big] e^{\otimes n}. 
\eeq
In most cases, it is more convenient to work with the locally defined Gaussian process 
	\beq\label{liftcir}\tilde s(z):= \Big[\sum_{j=1}^{d_n}b_j f_j(z)\Big]e^{-n\phi(z)/2},\eeq
	whose covariance kernel is given by
	 \beq\label{relat}\cov(\tilde s(z), \tilde s(w) )=\sum _{j=1}^{d_n}f_j(z)\overline{f_j(w)} e^{-n\phi(z)/2-n\phi(w)/2} =: \tilde F_n(z, w).\eeq
		Actually $\tilde s$ should be interpreted as the lift of the holomorphic section  $s$ of the  line bundle  to its equivariant function on the unit circle bundle  induced by $h$,  and $\tilde F_n(z, w)$ is the lift of the Bergman kernel which is the Szeg\"o kernel of level $n$ of the circle bundle. We refer to \cite[Section 1.2]{BSZ1}  for more   background  on circle bundle, Szeg\"o kernel, etc. By convention, throughout this article we will continue to refer to $\tilde F_n(z,w)$ as the Bergman kernel.


	

	
	On diagonal, the Bergman kernel has the following $C^\infty$ full expansion of the Tian-Yau-Zelditch Theorem \cite{D, T, Ze2}, 
	\begin{equation}\label{full}
		\tilde F_n(z,z)= \frac{n^m}{\pi^m}(1+a_1(z)n^{-1}+a_2(z)n^{-2}+\cdots),
	\end{equation}
	where all terms $a_j$ are computable and they are polynomials of curvatures, in particular, $a_1(z)$ is half of the scalar curvature of $\omega$.
	
	For the general setting of the compact \emph{almost complex symplectic manifolds},  under the normal coordinate and preferred frame, Shiffman-Zelditch proved that the Bergman kernel admits the $C^\infty$ full expansion about the rescaling limit around a point (\cite[Theorem 3.1]{SZ4}), 
	\begin{equation}\label{rescalinglimit12}
		\begin{split}&\tilde F_n(z_0+\frac u{\sqrt n},z_0+\frac v{\sqrt n})\\ =& \frac{n^m}{\pi^m} e^{u\cdot \bar v-\frac 12 |u|^2-\frac12|v|^2}\left(1+\sum_{r=1}^\ell n^{- r/2}p_r(u,v)+n^{-(\ell +1)/2}e_{n\ell}(u,v) \right),\end{split}
	\end{equation} 
	where each $p_r(u, v)$ is a polynomial, $e_{n\ell}$ and its derivatives are bounded on any compact subset in $\mathbb C^2$. 
	Actually, given any $\epsilon>0$,  for any $\ell\geq 0$, any $|u|, |v|\leq \log^2 n$ and any order of partial derivatives $\alpha,\alpha'$,  one has the uniform estimate (see \cite{DL} and also \cite[Theorem 4.2.1]{MM1}),
	\begin{equation}\label{rescalinglimit1234}
		\begin{split}&\Big|\frac{\partial ^{|\alpha|+|\alpha'|}}{\partial^{\alpha} u \partial^{\alpha'} \bar v }\Big(\frac {\pi^m} {n^m}\tilde F_n(z_0+\frac u{\sqrt n},z_0+\frac v{\sqrt n})\\ & -(1+\sum_{r=1}^\ell n^{- r/2}p_r(u,v))e^{u\cdot \bar v-\frac 12 |u|^2-\frac12|v|^2}\Big)\Big|=O(n^{-(\ell+1-\epsilon)/2}).\end{split}
	\end{equation} 
	Moreover, if we further restrict $u,v$ to stay in some compact set of $\mathbb{C}$ (independent of $n$), then we can remove the $\epsilon$ from the right  hand side of \eqref{rescalinglimit1234}. 
	
	As a corollary of \eqref{rescalinglimit12},  Shiffman-Zelditch proved that  the Bergman kernel and its derivatives exhibit exponential decay as $n$ large (\cite[Lemma 5.2]{SZ4}), 
	\beq\label{expdls}
	\|\tilde {F}_n(z, w)\|_{C^k} \leq Cn^{m+\frac{k}{2}} \exp \left(-c {d}_g(z, w) \sqrt{n}\right)
	\eeq
	where $c>0$ is some uniform constant, $C>0$ only depends on $k$, and the norms of the derivates of $\tilde {F}_n(z, w)$ is induced by the Riemann metric $g$.


	In the special case of polarized compact \kahler manifolds, one has $p_1=0$ (see, e.g., \cite[Theorem 4.1.22, Theorem 4.1.25 and Remark 4.1.26]{MM1} and \cite[Theorem 2.3]{LS}).   In particular, one has,
	\begin{equation}\label{rescalinglimit1222}
		\begin{split}&\tilde F_n(z_0+\frac u{\sqrt n},z_0+\frac v{\sqrt n})= \frac{n^m}{\pi^m} e^{u\cdot \bar v-\frac 12 |u|^2-\frac12|v|^2} \Big(1+n^{-1}\big(\frac 12s(z_0) \\ & +\frac18R(u, \bar u, u, \bar u)+\frac18R(v, \bar v, v, \bar v)-\frac14R(u, \bar v, u, \bar v)\big)+n^{-3/2}e_{n3} \Big),\end{split}
	\end{equation} 
	where the curvature tensor $R(s,\bar t, u,\bar v)=R_{j\bar k p\bar q}(z_0) s_j\bar t_k u_p \bar v_q$. 
	For example, in the case of  $(\mathcal O(1), h_{FS})\to (\CP,\omega_{FS})$,  the Bergman kernel reads, 
	$$\tilde F_n(z,w)=\frac{ n+1}{\pi}\frac{(1+z\bar w)^n}{\sqrt {(1+|z|^2)^n(1+|w|^2)^n}}.$$
	And its rescaling admits the full expansion (\cite[equation (33)]{LS}), 
	\begin{align*}\frac{1}{n}&\tilde F_n(\frac u{\sqrt n},\frac v{\sqrt n})=\frac 1 \pi e^{u\cdot \bar v-\frac 12 |u|^2-\frac12|v|^2} \Big(1+n^{-1}\big(1-\frac{(u\cdot \bar v)^2}2+\frac{|u|^4+|v|^4}4 \big)+ \cdots \Big).\end{align*}
	\subsection{Two classical models}
At the end of this section, we introduce two classical Gaussian random holomorphic fields that have been extensively studied.
	
	The first model is defined on complex projective space.  Recalling the orthonormal basis \eqref{basis1},   in the affine  coordinate, the Gaussian random holomorphic section  reads 
	$$\label{su2}p(z):=\sqrt{\frac{n+1}{\pi}} \sum_{j=0}^n b_j \sqrt{{n\choose j}}z^j  $$
	which is equivalent to the so-called \su random polynomials up to a common factor of magnitude \cite{BBL, H}. 

	The second model is defined on the complex plane. It is the Gaussian Entire Function (GEF) introduced in the introduction:
\beq\label{analyticm}
g(z) = \sum_{j=0}^{\infty} \frac{b_j}{\sqrt{j!}}\, z^j,
\eeq
where 
\(
\left\{ \frac{z^j}{\sqrt{j!}} \right\}_{j=0}^{\infty}
\)
forms an orthonormal basis of the Bargmann--Fock space of $\mathcal L^2$-integrable holomorphic functions
\(
\mathcal H\big(\mathbb C, \pi^{-1} e^{-|z|^2} d\ell \big),
\)
and $d\ell$ denotes the Lebesgue measure on $\mathbb C$.

	\section{ Random zeros}\label{randomzeros}
	
From now on, we restrict our attention to the case where \(M\) is a compact Riemann surface.  In this section, we review and derive basic properties of the correlation functions of zeros of Gaussian holomorphic sections over such surfaces.

	


We recall the Gaussian holomorphic section defined in \eqref{standardgauss} and its local expression \eqref{locald}.    It was proved in Theorem 1.1 of \cite{SZ1} that the expected density of its zeros satisfies
	\beq\label{zerodis}\mathbb E\Big( \frac 1 {n}\sum_{z_i\in Z_n}\delta_{z_i}\Big)\to \frac{\omega}{\pi}, \quad \mbox{as $n\to\infty$}\eeq
	in the sense of distribution. In particular, for \su random polynomials, the expected density is given by $ {\omega_{FS} }/\pi$, i.e., the uniform measure on $S^2$, exactly.
	
	
	Given  $Z_n$ in \eqref{defzn},
	we define another point process on $M^k:=\underbrace{M\times\cdots \times M}_k$ by
	\begin{equation}\label{znkdef}
		Z_n^{(k)}:=\sum_{\substack{z_{i_1}, \cdots,  z_{i_k}\in Z_n \\ \textrm{pairwise distinct}}}\,\,\delta_{(z_{i_1},\cdots, z_{i_k})}.
	\end{equation}
	Then the $k$-point correlation function $\rho_k$  of random zeros is defined as, 
	\beq\label{corre}\mathbb E[Z_n^{(k)}(B_1\times \cdots \times B_k)]=\int_{B_1\times \cdots \times B_k}\rho_k(z_1,\cdots, z_k)\frac{\omega(z_1)}{\pi}\wedge \cdots \wedge \frac{\omega(z_k)}{\pi}, \eeq
	where   $B_1, \cdots, B_k$ are Borel sets in  $M$.  In particular,  when $B_1=\cdots=B_k=B$,  it yields the $k$-th factorial moment of the number of points falling in $B$, 
	\beq \label{fmm}\mathbb E \left(\frac{(Z_n(B)!)}{(Z_n(B)-k)!}\right)=\int_{B^k}\rho_k(z_1,\cdots, z_k) \frac{\omega(z_1)}{\pi}\wedge \cdots \wedge \frac{\omega(z_k)}{\pi}.\eeq
	In fact, for any point process on Riemann surfaces, its correlation function, if exists,  can always be defined in the same manner as \eqref{corre}.

	We define the diagonal set of $M^k$
	$$ \Delta^{(k)}:=\{ (z_1,..,,z_k)\in M^k | \mbox{$z_i=z_j$ for some $i\neq j$}  \}.$$
	 Let $$\nabla=\frac \partial {\partial z}+A_n(z)$$ be  any connection of the line bundle $L^n$.  And thus locally we have
	 $$\nabla s=\Big[\sum_{j=1}^{d_n}b_j \Big(\frac {\partial f_j(z)}{\partial z}+A_n(z)f_j(z)\Big)\Big] dz \otimes e^{\otimes n}$$
	
	For  $(z_1,\ldots, z_k)\in M^k\backslash \Delta^{(k)}$, in the local coordinate, 
	the correlation function  $\rho_{k}(z_1,.., z_k)$ of zeros of Gaussian random holomorphic sections is given by the following Kac-Rice type formula (cf.~\cite[Theorem 11.2.1]{AT} and \cite[Section 2.2]{BSZ1}):
	\beq\label{rhord20}\begin{split}&\rho_k(z_1,\cdots, z_k)\\ =&\frac{\mathbb E\Big(\prod_{i=1}^k  \Big|\sum b_j \big(\frac {\partial f_j(z_i)}{\partial z}+A_n(z_i)f_j(z_i)\big)\Big|^2\Big|\sum b_j f_j(z_i)=0, i=1,..,k \Big)}{\det\left( F_n\left(z_i, z_{j}\right)\right)_{1 \leq i, j^{} \leq k}},\end{split}
	\eeq
where  $$  F_n\left(z, w\right): =\sum _{j=1}^{d_n}f_j(z)\overline{f_j(w)} .  $$

In \cite[Section 2.2]{BSZ1},  Bleher-Shiffman-Zelditch  further express the numerator of the conditional expectation in \eqref{rhord20} in terms of the Gaussian covariance matrix, using the standard Gaussian regression formula (see equations (42)–(43) in \cite{BSZ1}).
Both formulations are standard instances of the Kac-Rice formula.

	It's important to note that the choice of connection is arbitrary in our setting, since we are considering zeros of holomorphic sections, which are independent of the choice of derivative. This can be seen directly from the conditional expectation in formula \eqref{rhord20}: when conditioning on
 $$\sum_j b_j f_j(z_i) = 0,$$ we always have
\[
\sum_j b_j \big(\frac{\partial f_j(z_i)}{\partial z} + A_n(z_i) f_j(z_i)\big) = \sum_j b_j \frac{\partial f_j(z_i)}{\partial z},
\] which is independent the choice of of $A_n$. 
In practice, one may therefore choose $\nabla$ to be either the meromorphic flat connection $$\nabla^m:=\frac \partial {\partial z},$$   or 
	the smooth Chern connection $$\nabla^c:=\frac \partial {\partial z}-n\frac{\partial \phi}{\partial z},$$ or any other connection that is convenient for the computation. 
	
	Furthermore, formula \eqref{rhord20} can be rewritten as 
	\beq\label{rhord2}\rho_k(z_1,\cdots, z_k)=\frac{\mathbb E\left(|\widetilde {\nabla s}(z_1) |^2\cdots |\widetilde {\nabla s} (z_k)|^2\big|\tilde s(z_1)=\cdots=\tilde s(z_k)=0\right)}{\det\left(\tilde F_n\left(z_i, z_{j}\right)\right)_{1 \leq i, j^{} \leq k}},
	\eeq
	where \(\tilde F_n\) is the Bergman kernel as in \eqref{relat}, \(\tilde s\) is the Gaussian process defined in \eqref{liftcir} which is the lift of $s$ to its equivariant function on the circle bundle, and
	\beq \label{defdev}{\widetilde {\nabla s}}:=   \Big[\sum_{i=1}^{d_n}b_i \Big(\frac {\partial f_i(z)}{\partial z}+A_n(z)f_i(z)\Big)\Big]e^{-n\phi(z)/2}
	\eeq is the lift of $\nabla s$ on the circle bundle. 
    

    
As a direct consequence of the full expansion of the Bergman  kernel off-diagonal \eqref{rescalinglimit1222}, under the normal coordinate Bleher-Shiffman-Zelditch  proved the universal limit (\cite[Theorem 3.6]{BSZ1}),  
	\beq\label{universalk}\frac 1{n^{k}}\rho_k(z_0+\frac{z_1}{\sqrt n},..,z_0+ \frac{z_k}{\sqrt n})= \rho^\infty_k(z_1,..., z_k)+O\big({n}^{-1}\big), \eeq
	where $\rho^\infty_k(z_1,..., z_k)$ is the  correlation function for zeros of GEF defined in \eqref{analyticm}.  	In Corollary \ref{cords} below,  we will see the error term vanishes on the diagonal set  of $\mathbb{C}^k$,
	$$
	\{(z_1,\ldots, z_k)\in \mathbb{C}^k:z_i=z_j \mbox{ for some }i\neq j\},
	$$
	and	 $\rho_k$ can be extended smoothly to the whole $\mathbb C^k$.
We continue to denote this set by $\Delta^{(k)}$,  which will not cause any confusion.


	As a remark, in \cite{BSZ1, BSZ3}, Bleher-Shiffman-Zelditch proved \eqref{universalk} for the general compact \emph{almost complex symplectic manifolds}, where the error term in \eqref{universalk} is $O\big(n^{-1/2}\big)$ due to the non-zero value of the term $p_1$ in the rescaling limit \eqref{rescalinglimit12}. However, in the specific case of compact \kahler manifolds with polarized line bundes, one has $p_1=0$ in \eqref{rescalinglimit1222} (e.g. \cite[Remark 4.1.26]{MM1} and \cite[Theorem 2.3]{LS}). This, in turn, leads to the error term in \eqref{universalk} being of order\footnote{We would like to thank B. Shiffman for  clarifications regarding this matter.} $1/n$.

	
	For the limiting 2-point correlation function, one has an explicit formula (\cite[equations (106) and (107)]{BSZ1}),
	\beq \label{universal2}\rho_2^\infty(z, w) =H\big(\frac 12|z-w|^2\big), 
	\eeq
	where, as $t\to 0$,  \beq \label{Ht} \begin{split}H(t)&=\frac{\left(\sinh ^2 t+t^2\right) \cosh t-2 t \sinh t}{\sinh ^3 t}=t-\frac{2}{9} t^3+\frac{2}{45} t^5+O(t^7). \end{split}\eeq
	Then, one has \beq\label{rep}\rho_2^{\infty}(z,w)\to  0 \quad \mbox{as}\quad |z-w|\to 0,\eeq i.e., the nearby zeros are repulsive. 
	As a remark,  \eqref{universal2} and \eqref{Ht}  have been derived  in \cite{BBL, H} for the special case of \su random
	polynomials.

	Recall the Kac-Rice formula \eqref{rhord2}, it is important to note that difficulties arise near the diagonal $\Delta^{(k)}$, where both the numerator and the denominator tend to zero. In \cite{A, AL2}, instead of complex zeros of Gaussian holomorphic sections in $H^0(M, L^n)$ over compact Riemannian surfaces,  Ancona and Letendre studied the linear statistics of the real zeros  of Gaussian real sections in $\mathbb RH^0(M, L^n)$ over a compact real Riemann surface equipped with an antiholomorphic involution. The proofs in \cite{A, AL2} are based on the method of the divided differences, and the key ingredient is the exponential decay of the Bergman kernels together with its universal limit. The arguments can be applied for complex zeros of Gaussian holomorphic sections, as the real case and the complex case share exactly the same Bergman kernel.  Therefore, the results in \cite{A, AL2} hold for complex zeros. 
	
	We first have the following short-range behavior. 
	\begin{thm} \label{shortra}Under the normal coordinate around $z_0$,   
    for $n$ large enough and
    $(z_1,...,z_k)\notin \Delta^{(k)}$ where $|z_i|\leq \log^2 n$,  it holds that 
		\beq\label{localbound}\frac 1{n^k}\rho_k(z_0+\frac{z_1}{\sqrt n},..., z_0+\frac{z_k}{\sqrt n}) \leq  C \prod_{1\leq i < j\leq k}\min\big \{|z_i-z_j|^{2}, 1\big\},\eeq 	where $C>0$ is some uniform constant only depending on $k$. 
	\end{thm}
	This upper bound is  essentially the consequence of  the universal limit \eqref{universalk}  together with the following result for the correlation function of zeros of GEF   proved by Nazarov-Sodin (\cite[Theorem 1.3]{NS}), 
	\beq\label{criticalrhoinf}c^{-1} \prod_{1\leq i < j\leq k}\min \{|z_i-z_j|^{2}, 1\}\leq \rho^\infty _k(z_1,..., z_k) \leq c \prod_{1\leq i < j\leq k}\min \{|z_i-z_j|^{2}, 1\},\eeq
	where the $c$ is some positive constant only depending on $k$.
	
	
	Now, we outline the proof of Theorem \ref{shortra}, making slight modifications to those presented in \cite{A, MY} for the  completeness of the article, and the computations have their own interest.

	Given a smooth function $f: \mathbb{C} \rightarrow \mathbb{C}$ and a collection of pairwise distinct points $\underline z:=(z_1,..., z_k) \notin \Delta^{(k)}$, the divided difference of $f$ is  
	$$
	[f]_k\left(z_1, \ldots, z_k\right) \stackrel{\text {}}{:=} \sum_{j=1}^k \frac{f\left(z_j\right)}{\prod_{j^{\prime} \neq j}\left(z_j-z_{j^{\prime}}\right)} .
	$$
	By the definition of the divided differences, we have:
	\beq\label{mmatrix}
	\begin{bmatrix}
		f\left(z_1\right) \\
		\vdots \\
		f\left(z_k\right)
	\end{bmatrix}=T(\underline  z)\begin{bmatrix}
		{[f]_1\left(z_1\right)} \\
		\vdots \\
		{[f]_k\left(z_1, \ldots, z_k\right)}
	\end{bmatrix},
	\eeq 
	where $T(\underline z)$ is a $k\times k$ lower-triangular matrix whose  $(i,j)$-th entry is given by
	$$
	\prod_{\ell=1}^{j-1} (z_i-z_{\ell}),
	$$
	where an empty product is set to be 1. Thus we have 
	$$\operatorname{det} T(\underline z)=\prod_{1 \leq i<j \leq k}\left(z_j-z_i\right).$$ 
	Note that by definition $$[f]_k(z,.., z)=\frac{f^{(k-1)}(z)}{(k-1)!}. $$
	We refer to \cite[Section 5]{AL} for a comprehensive introduction to divided differences. 
	
	For any fixed $z_0\in M$, we choose a normal coordinate such that $z_0=0$ and consider the points $({z_1}/{\sqrt n},\cdots, {z_k}/{\sqrt n})$ in a small neighborhood around $z_0$. 
	We denote $\hat z_i=z_i/\sqrt n$ and $\widehat{{\underline z}}=(\hat z_1,\cdots,\hat z_n)$. Recall the Gaussian process $\tilde s$ in \eqref{liftcir}.  We choose the connection $\nabla=\frac{\partial }{\partial z}- \frac{n\phi}2\frac{\partial\phi}{\partial z}$.  By \eqref{defdev}, we have  $$\widetilde {\nabla s}=    [\sum_{i=1}^{d_n}b_i \big(\frac {\partial f_i(z)}{\partial z}-\frac{n\phi}2\frac{\partial\phi}{\partial z}f_i(z)\big)]e^{-n\phi(z)/2}=[\big(\sum_{i=1}^{d_n}b_i f_i(z)\big)e^{-n\phi(z)/2}]'
	=\tilde s'.$$ By the Kac-Rice formula in \eqref{rhord2}, we have   \beq\label{d}
	\frac 1{n^{k}}\rho_k(\frac{z_1}{\sqrt n},.., \frac{z_k}{\sqrt n})=\frac{\mathbb{E}\left[\prod_{j=1}^k  |\big(\tilde s(\widehat{z}_j)\big)'|^2 \mid  \tilde s(\widehat{z}_1)=\cdots  =\tilde s(\widehat{z}_k)= 0\right]}{\operatorname{det}\left[\operatorname{Cov}\left( \tilde s\left(\widehat{z}_j\right)\right)_{j=1}^k\right]}.
	\eeq
	
	The configuration $\{\hat z_1, ...,\hat  z_k\}$ will induce a partition of $\{1, ..., k\}$ as follows. We draw an edge between two vertices $p$ and $q$ if $|\hat z_p-\hat z_q|\leq 1/\sqrt n$. This leads to a graph with vertex set $\{1,\ldots, k\}$, whose connected components induce a partition $\mathcal P=\{I_1, .., I_\ell\}$ such that $\cup_{i=1}^\ell I_i=\{1, ..., k\}$ and $I_i\cap I_j=\emptyset$ for $i\neq j$.  We denote $I_i=\{{i_1}, ..., {i_{n_i}}\}$, and 
	$\widehat{ \underline{z}}_{I_i}=(\hat z_{i_1}, ...,\hat z_{i_{n_i}}) $, $\tilde s (\underline {\hat z}_{I_i})=(\tilde s(\hat z_{i_1}), ...,\tilde s(\hat z_{i_{n_i}}))^T $ and $[\tilde s](\underline {\hat z}_{I_i})=([\tilde s]_1(\hat z_{i_1}), ...,[\tilde s]_{n_i}(\hat z_{i_1}, ..., \hat z_{i_{n_i}}) )^T$.  
	Then  by \eqref{mmatrix}, we have the divided differences equation, 
	$$
	\begin{bmatrix}
		\tilde s\left(\underline {\hat z}_{I_1}\right) \\
		\vdots \\
		\tilde s\left(\underline {\hat z}_{I_\ell}\right)
	\end{bmatrix}=
	\begin{bmatrix}
		T(\underline {\hat z}_{I_1}) &  &  \\
		& \ddots &  \\
		&  &   T(\underline {\hat z}_{I_\ell})\\
	\end{bmatrix}
	\begin{bmatrix}[\tilde s]\left(\underline {\hat z}_{I_1}\right) \\
		\vdots \\
		[\tilde s]\left(\underline {\hat z}_{I_\ell}\right)
	\end{bmatrix}
	:=\mathbf{T}(\underline z) 	\begin{bmatrix}[\tilde s]\left(\underline {\hat z}_{I_1}\right) \\
		\vdots \\
		[\tilde s]\left(\underline {\hat z}_{I_\ell}\right)
	\end{bmatrix}.
	$$
	By the construction of the partition, we have 
	\begin{equation}
		\begin{split} 
			\mathcal T(\underline z):=&
			\abs{\det(\mathbf{T}(\underline z))}^2\\
			=&	
			|\det (  T(\underline {\hat z}_{I_1}))\cdots   \det(T(\underline {\hat z}_{I_\ell})) |^2= n^{-\sum_{i=1}^\ell n_i(n_i-1) / 2} { \prod_{i<j,i\sim j}|z_i-z_j|^2}.
		\end{split} 
	\end{equation}
	Here we write the notation $i\sim j$ to mean that $i$ and $j$ are in the same connected component. 
	For the denominator,  one has 
	\beq\label{b}
	\operatorname{det}\left[\operatorname{Cov}\left( \tilde s\left(\widehat{z}_j\right)\right)_{j=1}^k\right]=\mathcal T(\underline z) \operatorname{det}\left[\operatorname{Cov}\begin{bmatrix}[\tilde s]\left(\underline {\hat z}_{I_1}\right) \\
		\vdots \\
		[\tilde s]\left(\underline {\hat z}_{I_\ell}\right)
	\end{bmatrix}\right].
	\eeq
	For the numerator,  conditioned on  $\tilde s(\hat z_i)=0$ for all $i=1, .., k$, one can show that (e.g., \cite[Lemma 3.5]{MY}), 
	\beq \label{f}
	\begin{split}
		\prod_{j=1}^k |\big( \tilde s(\widehat{z}_j)\big)'|^2&=n^{-k} \prod_{i=1}^\ell \left(\prod_{q=1}^{n_i} \left( \left| [\tilde s]_{n_i+1}(\underline{\hat{z}}_{I_i}, \hat{z}_{i_q})\right|^2 \prod_{i_r:\, i_r\neq i_q\in I_i}\left|\widehat{z}_{i_q}-\widehat{z}_{i_r}\right|^2\right)\right) \\
		& =n^{-k}\mathcal T(\underline z)^2  \prod_{i=1}^\ell \left(\prod_{q=1}^{n_i}\left| [\tilde s]_{n_i+1}(\underline{\hat{z}}_{I_i}, \hat{z}_{i_q})\right|^2 \right)\\
		& ={ \prod_{i<j, i\sim j} |z_i-z_j|^4 }\prod_{i=1}^\ell \left(\prod_{q=1}^{n_i}\left| \frac{[\tilde s]_{n_i+1}(\underline{\hat{z}}_{I_i}, \hat{z}_{i_q})}{n^{n_i/2}}\right|^2 \right).
	\end{split}
	\eeq
	Set 
	\beq\label{g}
	\begin{aligned}
		\mathcal N_n:=&\mathbb E \left( \prod_{i=1}^\ell \prod_{q=1}^{n_i}\left| \frac{[\tilde s]_{n_i+1}(\underline{\hat{z}}_{I_i}, \hat{z}_{i_q})}{n^{n_i/2}}\right|^2 
		[\tilde{s}](\underline{\hat{z}}_{I_1})=\mathbf{0},\ldots, 	
		[\tilde{s}](\underline{\hat{z}}_{I_{\ell}})=\mathbf{0}	 \right),\\
		\mathcal D_n:=&n^{-\sum_{i=1}^\ell n_i(n_i-1) / 2} {\operatorname{det}\left[\operatorname{Cov}\begin{bmatrix}[\tilde s]\left(\underline {\hat z}_{I_1}\right) \\
				\vdots \\
				[\tilde s]\left(\underline {\hat z}_{I_{\ell}}\right)
			\end{bmatrix}\right]}.
	\end{aligned} \eeq
	Then equations \eqref{d}-\eqref{f} will imply
	\begin{equation*}
		\frac 1{n^{k}}\frac{\rho_k(\frac{z_1}{\sqrt n},.., \frac{z_k}{\sqrt n})}{ \prod_{i<j, i\sim j} |z_i-z_j|^2}=\frac{\mathcal N_n}{\mathcal D_n}.
	\end{equation*}
	Note that in \eqref{g}, for the numerator $\mathcal N_n$, we rewrite the condition $\tilde s(\hat{z}_1)=\cdots =\tilde s(\hat{z}_{k})=0$ in its equivalent form in terms of divided differences $[\tilde{s}](\underline{\hat{z}}_{I_{i}})=\mathbf{0}$, $i=1, .., \ell$, so that the covariance matrix of  $[\tilde{s}](\underline{\hat{z}}_{I_{i}})$ is the one in the denominator $\mathcal D_n$.
	
	Ancona proved that, 
	the denominator $\mathcal D_n$ has positive uniform lower bound (\cite[Proposition 4.22]{A})  and the numerator $\mathcal N_n$ has uniform upper bound (\cite[Proposition 4.23]{A})\footnote{Note that the results in \cite{A} are stated for $|z_i|<C\log n$, but they extend to 
		$|z_i|<\log^2 n$ immediately without modification, since the key ingredient \eqref{rescalinglimit1234} holds in both cases.}, and both bounds only depend on the pattern of the partition $\mathcal P$. Therefore, $\mathcal N_n/\mathcal D_n$ is bounded from above by some constant only depending on the pattern of the partition $\mathcal P$. As there are only finite many partitions of $\{1,..., k\}$, $N_n/\mathcal D_n$ can be bounded from above uniformly. This,  together with the fact that there exists some positive constant $C$ depending only on $k$ such that 
	$$
	C^{-1} \prod_{1\leq i < j\leq k}\min\big \{|z_i-z_j|^{2}, 1\big\} \leq \prod_{i<j, i\sim j} |z_i-z_j|^2 \leq 
	C  \prod_{1\leq i < j\leq k}\min\big \{|z_i-z_j|^{2}, 1\big\}, 
	$$
	will complete the proof of Theorem \ref{shortra}. 
	
	For the configuration $\{ z_1, ..., z_k\}$, we draw an edge between two vertices $p$ and $q$ if $| z_p-z_q|\leq 1$, which induces a partition $\mathcal P=\{I_1, .., I_\ell\}$  of $\{1, ..., k\}$ in the same manner as before. 
For each $I_i=\{{i_1}, ..., {i_{n_i}}\}$, we denote ${ \underline{z}}_{I_i}=( z_{i_1}, ..., z_{i_{n_i}}) $. 	
	
	Let $\tilde g(z):=g(z)e^{-|z|^2/2}$ be the weighted  GEF.  Following the same computations as above,
	set
	\begin{equation*}
		\begin{aligned}
			\mathcal N	=&\mathbb E\left(\prod_{i=1}^\ell \prod_{i_s=1}^{n_i}\left| {[\tilde g]_{n_i+1}(\underline{{z}}_{I_i}, {z}_{i_s})}{}\right|^2 |
			[\tilde g]\left(\underline { z}_{I_1} \right)=\mathbf{0},\ldots, 	[\tilde g]\left(\underline { z}_{I_\ell}\right)=\mathbf{0}
			\right),
		\end{aligned}
	\end{equation*}	
	and
	\begin{equation}
		\begin{aligned}
			\mathcal D=	{\operatorname{det}\left[\operatorname{Cov}\begin{bmatrix}[\tilde g]\left(\underline { z}_{I_1}\right) \\
					\vdots \\  [\tilde g]\left(\underline { z}_{I_\ell}\right)
				\end{bmatrix}\right]}
		\end{aligned}.
	\end{equation}
	Then we have
	\begin{equation}\label{alfer}
		\frac{\rho_k^\infty ({z_1}{},.., {z_k}{})}{  \prod_{i<j, i\sim j} |z_i-z_j|^2  }=\frac{\mathcal N}{\mathcal D}.
	\end{equation} 
	By Nazarov-Sodin's result \eqref{criticalrhoinf},  $\mathcal N/\mathcal D$ has  uniform upper and below bounds on  the whole complex plane $\mathbb C$. Note that Ancona also proved that $\mathcal D$ has some positive uniform lower bound (\cite[Proposition 4.30]{A}).
	
	Lemma 4.6 in \cite{A} implies that the covariance of the divided difference can be expressed as the linear combinations of Bergman kernel and its derivatives. Therefore, the $C^\infty$ full  expansion of the Bergman kernel off-diagonal  in \eqref{rescalinglimit1234} will imply the full  expansion of $\mathcal N_n$ and $\mathcal D_n$ with leading order terms $\mathcal N$ and $\mathcal D$, respectively (e.g., \cite[Proposition 4.35]{A}). By \eqref{rescalinglimit1234}, for any compact set $K$, we have
	\begin{equation}\label{nnn}
		\mathcal N_n=\mathcal N+O(n^{-1}),\,\, 
		\mathcal D_n=\mathcal D+O(n^{-1}), 
	\end{equation}
	where the error is	uniform for all $z_1,\ldots, z_k\in K$.

	As we mentioned before,  both $\mathcal D_n$ and $\mathcal D$ have positive uniform lower bounds. This together with \eqref{nnn} implies
	$$\mathcal N_n/\mathcal D_n=\mathcal N/\mathcal D+O(n^{-1}),$$
	which gives the following improvement of \eqref{universalk}, 
	$$\frac 1{n^{k}}\frac{\rho_k(\frac{z_1}{\sqrt n},.., \frac{z_k}{\sqrt n})}{\prod_{i<j, i\sim j} |z_i-z_j|^2 }=
	\frac{\rho_k^\infty ({z_1}{},.., {z_k}{})}{ \prod_{i<j, i\sim j} |z_i-z_j|^2 }+O\big( n^{-1}\big),$$
	where the error term is smooth for large $n$. This leads to  \begin{cor}\label{cords}
		The error term  in  \eqref{universalk} is smooth, vanishes on diagonal $\Delta^{(k)}$, and decays with an order of $ n^{-1}$ on any compact subset in $\mathbb C^k$.    \end{cor}

The next result establishes the splitting property of correlation functions.  Given a configuration of distinct points $(z_1,..,z_k)\in M^k\backslash \Delta^{(k)}$,  for two nontrivial subsets $I, J\subset \{1, 2,.., k\}$ such that $I\cap J=\emptyset$,   the distance between $I$ and $J$ is $$d_g(I, J) :=\min_{i \in I, j\in J} d_g(z_i, z_j).$$
	Given a partition of  $\{1, 2,.., k\}$, denoted as $\mathcal P=\{I_1,.., I_\ell\}$, we define the minimal distance between all clusters (a cluster corresponds to a subset in the partition of $\{1,\ldots, k\}$), 
	$$\eta:=\min_{1\leq i\neq j\leq \ell} d_g(I_i, I_j).$$
	Given a nontrivial subset $I  \subset \{1, 2,.., k\}$,  we again denoted by $\underline{z}_{I}=\left(z_i\right)_{i \in I}$ as a cluster of variables. Then we have 
	
	\begin{thm}\label{longdr}
		Given $(z_1,..., z_k) \in M^k\backslash \Delta^{(k)}$ and a partition $\mathcal P=\{I_1,.., I_\ell\}$ of $\{1,..., k\}$ such that the minimal distance $\eta$ between different clusters is at least $\log^2n/\sqrt{n}$, then   the correlation function exhibits the following splitting property,
		\beq\label{mainsplit2}
		\rho_k(z_1,.., z_k)=\prod_{I_i \in \mathcal{P}} \rho_{|I_i|}\left(\underline{z}_{I_i}\right)+O\big(n^{-\infty}\big),
		\eeq  
		where $n^{-\infty}$ is an error term that decays faster than $n^{-c}$ for any $c>0$. 
	\end{thm}
	For the case  $k=2$,  \eqref{mainsplit2}  has been  derived by Shiffman-Zelditch (\cite[Theorem 3.16]{SZ6}).
	Based on the method of divided differences, \eqref{mainsplit2} has been established in Proposition 3.2 of \cite{A}  for real zeros  of Gaussian real sections in $\mathbb RH^0(M, L^n)$ (see also \cite[Proposition 2.26]{AL2}).  Theorem \ref{longdr} for the complex case can be derived by  the same proofs as the real case, as the proofs in both cases are essentially a result of the exponential decay of the Bergman kernel and its derivatives (recall \eqref{expdls}).  
	
	For GEF,  the splitting property has been proved by Nazarov-Sodin (\cite[Theorem 1.2]{NS}). In particular, their result  implies that there exists some constant $c>0$ such that 
	\beq\label{splitforc}
	\rho_k^\infty(z_1,.., z_k)=\prod_{I_i \in \mathcal{P}} \rho^\infty_{|I_i|}\left(\underline{z}_{I_i}\right)+O\big(e^{-c\delta}\big), 
	\eeq   
	where $\delta$ is  the minimal distance between different  clusters on $\mathbb{C}$.

	%
	
	
	In \cite{AL},  by the same method of divided differences, Ancona-Letendre derived the similar results to   Theorems \ref{shortra} and \ref{longdr} for the zeros of stationary Gaussian processes with rapid decay covariance kernels. In \cite{MY}, Michelen-Yakir established  analogous results to Theorems \ref{shortra} and \ref{longdr} for the special case of $\operatorname{SU(2)}$ random polynomials.
	
	As a consequence of Theorems \ref{shortra} and   \ref{longdr}, 
	one has the following smoothness and global boundedness of the correlation functions (see also \cite[Theorem 4.7  and Theorem 3.1]{A}), although we will not use this theorem in proving the smallest distances. 
	\begin{thm}\label{unofor}
		The correlation function $\rho_k(z_1, ..., z_k)$ of complex zeros of Gaussian random holomorphic sections can be extended to a smooth function  which vanishes on diagonal for any  $n$ sufficiently large.  Moreover, there exists  $C>0$, for $n$ sufficiently large, one has 
		$$\label{globds}
		\frac 1{n^k}\rho_k(z_1,.., z_k)<C.
		$$ 
	\end{thm}

	As a remark, the crucial point is that Theorem \ref{shortra},  Theorem \ref{unofor} and Corollary \ref{cords} do not hold for higher dimensions.  In \cite{BSZ1},  Bleher-Shiffman-Zelditch studied the simultaneous zeros of $1\leq \ell \leq m$ Gaussian random holomorphic sections over complex $m$-dimensional \kahler manifolds. Again, as a direct consequence of the full expansion of Bergman kernel off-diagonal and the Kac-Rice formula,   the $k$-point correlation function of  zeros still has the  (pointwise) universal limit outside the diagonal set (\cite[Theorem 3.6]{BSZ1}), 
	\beq\label{universalk345ggod}\frac 1{n^{k\ell }}\rho_k(z_0+\frac{z_1}{\sqrt n},..,z_0+ \frac{z_k}{\sqrt n})= \rho^\infty_{k, \ell, m}(z_1,..., z_k)+O\big({n}^{-1}\big). \eeq
	Here,  $\rho^\infty_{k,\ell,m }(z_1,..., z_k)$ is a universal  function.

	However,  it is important to note that, unlike the $1$-dimensional case as in Corollary \ref{cords}, the error term $O(n^{-1})$ only holds on any compact subsets $K\subset (\mathbb C^m)^{ (k)}$, where 
	$$ (\mathbb C^m)^{(k)}:=\{(z_1,\ldots, z_k) \in  (\mathbb C^m)^k, z_i\neq z_j, \forall 1\leq i\neq j\leq k \}.$$
	This is because, for higher dimensional cases,   $\rho_k(z_1,...,z_k)$ may blow up on diagonal. The same behavior applies to $\frac 1{n^{k\ell }}\rho_k(z_0+\frac{z_1}{\sqrt n},..,z_0+ \frac{z_k}{\sqrt n})$ and  $\rho^\infty_{k, \ell, m}(z_1,..., z_k)$.
	
	For example, by recalling the Taylor expansion \eqref{mdim} for the limiting 2-point correlation of discrete zeros with $\ell = m$, 
one obtains distinct short-distance behaviors depending on the dimension of the complex manifold 
(see the discussion following Theorem 1.3 in \cite{BSZ4}).   For $m\geq 3$, we have  
	$$\rho^\infty_{2,m,m}(u,v)\to \infty\,\,\, \mbox{as}\,\,\, u\to v,  \,\,$$
	i.e., nearby zeros are attractive for $m\geq 3$, and thus Theorem \ref{shortra} and Theorem \ref{unofor} should both fail. For $m=2$, we have 
	$$\rho^\infty_{2,2,2}(u,v)\to \frac 34\,\,\, \mbox{as}\,\,\, u\to v,  \,\,$$
	i.e., nearby zeros are neutral, and thus Theorem \ref{shortra} should also fail, but it is quite possible that Theorem \ref{unofor} holds for $m=2$. 
	
	In all cases, we may expect that Theorem \ref{longdr} always holds, as it is essentially the consequence of the exponential decay of the Bergman kernel. 
	
	

	
	

	\section{Proof of Theorem \ref{main1}}\label{proofofmain}
	Now we prove Theorem \ref{main1}.  Fix any $a>0$ and set $$a_n:=n^{-{3}/4}a.$$
	We first prove  the following Lemma \ref{lunis} which provides the limit of  the integration of the 2-point correlation function over a small geodesic ball with radius $a_n$. This limit determines the  intensity of the limiting Poisson point process of the smallest distances between random zeros on Riemann surfaces. 	Given a point $z\in M$, we denote $B_{a_n}(z)$ as the geodesic ball centered at $z$ with radius $a_n$.

	\begin{lem}\label{lunis}
		The integration of the 2-point correlation function over the small geodesic ball of radius $a_n$ has  the universal limit, 
		\beq \label{maintwolimit} \lim_{n\to\infty} \int_{B_{a_n}(z_0)} \rho_2(z_0, w)\frac{\omega(w)}{\pi} =\frac {a^4}4, \mbox{ uniformly in }z_0\in M.
		\eeq
		
	\end{lem}
	
	The proof of Lemma~\ref{lunis} relies on the rescaling estimates in \kahler normal coordinates for correlation functions and the volume form, many of which also appear in the proof of Lemma~3.2 in \cite{Z}.
	
		 Let $z_0=0$ be the origin of the normal coordinate.    In what follows, fix a local coordinate chart around $z_0$ and consider the rescaling map $$
\phi_n:z\mapsto z/\sqrt{n}.
$$
For any differential form $\alpha$ on $M$, we write
\(
\alpha(z/\sqrt{n}):=(\phi_n^*\alpha)_z.
\)
In other words, $\alpha(z/\sqrt{n})$ denotes the pullback of $\alpha$ under the rescaling map $\phi_n$ in these local coordinates.
	Recall \eqref{taylorforkahler},  for sufficiently large $n$, around $z_0$, the \kahler form  satisfies
	\beq\label{volumeestimate}
	\omega(z/\sqrt{n})= {n}^{-1}\Big(1+O\left(n^{-{1}}|z|^2\right)\Big)d\ell_z, 
	\eeq
	where we denote  $$d\ell_z:=\frac{\sqrt{-1}}2dz\wedge d\bar z.$$ This implies that the geodesic distance between $z_0+ z/{\sqrt n}$ and $z_0$ satisfies 
	\beq\label{distanceest}d_g(z_0, z_0+\frac z{\sqrt n})=\frac{|z|}{\sqrt n}+O\big(n^{-{3/2}}|z|^3\big).\eeq



	By  the small geodesic distance estimate \eqref{distanceest},  the volume estimate \eqref{volumeestimate}, the convergence statement \eqref{universalk}, Corollary \ref{cords} and the Taylor expansion for the 2-point correlation function in \eqref{universal2} and \eqref{Ht},   we have
	\beq\label{computetwo}\begin{split}&\lim_{n\to\infty}\int_{w\in B_{a_n}(z_0)} \rho_2(z_0,w)\frac{\omega(w)}\pi \\  = & \pi^{-1} \lim_{n\to\infty}\int_{|w|\leq a_n}\rho_2(0,w){\omega(w)}\\ =&\pi^{-1}\lim_{n\to\infty} \int_{|z|\leq  n^{1/2}a_n}\rho_2(0, \frac z{\sqrt n})\omega\left(\frac z{\sqrt n}\right)\\ =& \pi^{-1}\lim_{n\to\infty} n  \int_{|z|\leq  n^{1/2}a_n}[\frac1{n^2}\rho_2(0, \frac z{\sqrt n})]d\ell_z \\=&
		\pi^{-1}\lim_{n\to\infty} n \int_{ |z|\leq  n^{-1/4}a}[\rho_2^\infty(0,  z)+O(n^{-1})]d\ell_z 
		\\ =&  \pi^{-1}\lim_{n\to\infty} n  \int_{ |z|\leq  n^{-1/4}a}\rho_2^\infty(0, z)d\ell_z
		\\=& \pi^{-1}\lim_{n\to\infty} n  \int_{ |z|\leq  n^{-1/4}a}\frac 12 |z|^2d\ell_z  \,\,\\=&\frac 14 a^4,\end{split}\eeq
	which completes the proof of Lemma \ref{lunis}.  

	
	


	\subsection{An auxiliary point process}
	Recall the definition of $\mathcal S_n$ in \eqref{alldistanorm}. Given a measurable subset $U\subseteq M$ and a bounded measurable set $A\subset [0,\infty)$, to prove Theorem \ref{main1}, 
	it suffices to show that
		\begin{equation}\label{snpo-0}
		\mathcal S_n(A\times U) \Rightarrow  \mbox{Poisson}\left(
		\int_A \frac{x^3}{2}\mathrm{d}x 
		\int_{U}  \frac{\omega}{\pi}\right) \quad\mbox{in distribution}. 
	\end{equation}
	
We will give a detailed proof for $A$ of the form $[0,a]$ where $a>0$. Other forms of $A$  can be treated in the same way. In other words, our goal is to establish that 
	\begin{equation}\label{snpoi}
		\mathcal S_n([0,a]\times U) \Rightarrow  \mbox{Poisson}\left(\frac  { a^4}8 
		\int_{U}  \frac{\omega}{\pi}\right) \quad\mbox{in distribution}. 
	\end{equation}
	Given the set of random  zeros $Z_n=\{z_1,\cdots, z_n\}$, we define  the following set containing all pairs of zeros whose distances are within $a_n$,
	\beq\label{ira}\mathcal I_n:=\Big\{ (i,j):\, 1\leq i< j\leq n, \, d_g(z_i,z_j)<a_n\Big\}.\eeq
	By definition, we have 
	$$\label{allequ}  
	\abs{\mathcal I_n}=\mathcal S_n([0,a]\times M).
	$$
	We define a subset of $\mathcal{I}_n$ as 
	\begin{equation}\label{tildein}
		\begin{split}	\tilde{\mathcal I}_n:=\Big\{&
			(i,j)\in \mathcal I_n: \nexists (k,\ell)\in \mathcal I_n, (k,\ell)\neq (i,j), \\ & d_g(\{z_i,z_j\}, \{z_k,z_{\ell}\})\leq 4\log^2 n/\sqrt n
			\Big\},\end{split}
	\end{equation}
	where  we define the distance between two pairs of points as
	\beq\label{distancebetweenpairs}
	d_g(\{z_i,z_j\}, \{z_k,z_{\ell}\})=\min_{x\in \{z_i,z_j\}, y\in  \{z_k,z_{\ell}\} }d_g(x,y).
	\eeq
	Recall that $\hat{z}_{ij}$ is a point uniformly chosen from $\{z_i,z_j\}$.  We define an auxiliary random process as 
	\begin{equation}\label{altessssds} \mathlarger\chi_n:=\sum_{(i,j)\in    \tilde{\mathcal I}_n} \delta_{\hat{z}_{ij} }.   \end{equation}
	We first have the following
	\begin{lem}\label{nssn}As $n\to\infty$, one has the convergence,  $$
		\abs{\tilde{\mathcal I}_n}
		\stackrel{\text { }}- \abs{\mathcal I_n}\to0 \,\, \mbox{ in probability. }$$
	\end{lem}
	Once Lemma \ref{nssn} is established, to prove \eqref{snpoi},
	it is sufficient to prove 	$$\label{chipoi}
	\mathlarger {\chi}_n(U)
	\Rightarrow  \mbox{Poisson}\left(\frac  {a^4}8 
	\int_{U}  \frac{\omega}{\pi}\right)\quad \mbox{in distribution},
	$$
	which is   implied by    \beq\label{alterfactor}
	\lim_{n\to\infty}\mathbb E\Big(\frac{ \mathlarger{\chi }_n(U) !}{(\mathlarger{\chi}_n(U)-k)!} \Big) =\Big(\frac   {a^4}8 
	\int_{U}  \frac{\omega}{\pi}\Big)^k\quad\mbox{ for any $k\geq 1$.}\eeq  
	Let $\tilde\rho_k(z_1,.., z_k)$ denote the $k$-point correlation function of the auxiliary point process $ \mathlarger{\chi}_n$. By the definition of the point correlation function in \eqref{fmm}, \eqref{alterfactor} is equivalent to  \beq \label{tilderhok}\lim_{n\to\infty}\int_{{U}^k}\tilde \rho_k(z_1, ..., z_k) \frac{\omega(z_1)}\pi\wedge\cdots\wedge \frac{\omega(z_k)}\pi=\Big(\frac { a^4}8 \int_{U} \frac{\omega}{\pi} \Big)^k.
	\eeq
	The proof of \eqref{tilderhok} will be given in the Section \ref{sec:pfmain1}.
	\begin{proof}[Proof of Lemma \ref{nssn}]  
		By the definitions of $\mathcal I_{n}$ and $\tilde {\mathcal I}_{n}$, we have
		$$\label{snbd1}
		0\leq |\mathcal I_{n}|-|\tilde {\mathcal I}_{n} |   \leq |D|,
		$$
		where\footnote{The reason of separating the analysis into two sets is to match the definition of correlation functions: $D_1$ and $D_2$ can be handled by $\rho_3$ and $\rho_4$, respectively.} the set $D:=D_1\cup D_2$, and $D_1, D_2$ are given by
		\begin{equation*}
			\begin{split}
				D_1:=\Big\{&(z_1,z_2,z_3)\in Z_{n}^{(3)}: d_g(z_1,z_2)\leq a_n, d_g(z_1,z_3)\leq a_n\Big\}, \\
				D_2:=\Big\{&(z_1,z_2,z_3,z_4)\in Z_{n}^{(4)}: d_g(z_1,z_2)\leq a_n, d_g(z_3,z_4)\leq a_n, \\&d_g(z_1,z_3)<5\log^2 n/\sqrt n\Big\}, 
			\end{split}
		\end{equation*}
		where $Z_n^{(k)}$ is defined in \eqref{znkdef}. Indeed, if $(i,j)\in \mathcal I_n\backslash  \tilde{\mathcal I}_n$, then there exists another pair $(k,\ell)\in \mathcal I_n$ such that the distance between 
        $\{z_i,z_j\}$ and $\{z_k, z_{\ell}\}$ is bounded by $4\log^2n/\sqrt{n}$. The set $D_1$ corresponds to the case where there is a nontrivial overlap between these two sets, while $D_2$ accounts for the scenario where all four points $z_i,z_j,z_k,z_{\ell}$ are pairwise disjoint. In the second case, we also have,  for $n$  large enough,
        $$
        \max_{x\in \{z_i,z_j\}, y\in  \{z_k,z_{\ell}\} }d_g(x,y)\leq \log^2 n/\sqrt{n}+2a_n \leq 5\log^2 n/\sqrt{n}.
        $$
       
		To prove that the integer-valued random variable $|D|$ converges to 0 in probability, it is sufficient to prove that both $\E|D_1|$ and $\E|D_2|$ tend to 0. 
		
		By the definition of the correlation function \eqref{corre}, we first have  
		$$\mathbb E|D_1|= \int_{M}  \frac{\omega(z_1)}{\pi}  \int_{d_g(z_1, z_2)<a_n, d_g(z_1, z_3)<a_n} \rho_{3} \left(z_1, z_{2}, z_{3}\right)  \frac{\omega(z_2)}{\pi} \wedge  \frac{\omega(z_3)}{\pi}. $$
		For any  $z_1$,  we can choose a normal coordinate around $z_1$ such that $z_1=0$. Then  using the small geodesic distance estimate \eqref{distanceest}  and a change of variables, we get  
		\begin{equation*}\begin{split}
				\mathbb E|D_1|\leq &C\int_{M}\left( \int_{|z_2|\leq a_n, |z_3|\leq a_n } \rho_{3} \left(0, z_{2}, z_{3}\right)  \frac{\omega(z_2)}{\pi} \wedge  \frac{\omega(z_3)}{\pi} \right)  \frac{\omega(z_1)}{\pi} \\
				=&C \int_{M}\left( \int_{|z_2|\leq n^{1/2}a_n, |z_3|\leq n^{1/2}a_n} \rho_{3} \left(0, \frac{z_{2}}{\sqrt n},\frac{ z_{3}}{\sqrt n}\right) \omega\left({\frac {z_2}{\sqrt{n}}}\right)\wedge \omega\left({\frac {z_3}{\sqrt{n}}}\right) \right) \omega(z_1).
		\end{split} \end{equation*}
		By the  volume estimate \eqref{volumeestimate}  and  the upper bound \eqref{localbound}, we further have
		\beq\label{differered} \begin{aligned}
			\E|D_1|	\leq &Cn^{-2} \int_M \left( \int_{|z_2|\leq n^{1/2}a_n, |z_3|\leq n^{1/2}a_n} \rho_{3} \left(0, \frac{z_{2}}{\sqrt n},\frac{ z_{3}}{\sqrt n}\right)  {d} \ell_{z_{2}} \wedge {d}\ell_{z_{3}}  \right) \omega(z_1)\\
			\leq &Cn^{}  \int_{|z_2|\leq n^{1/2}a_n, |z_3|\leq n^{1/2}a_n} |z_2|^2 |z_3|^2|z_2-z_3|^2  {d} \ell_{z_{2}}\wedge  {d}\ell_{z_{3}} 
			=O(n^{-3/2}).
		\end{aligned}\eeq
		Similarly we can bound $\E|D_2|$ by
		\begin{equation*} \begin{split}	
				&\int_{M}  \frac{\omega(z_1)}{\pi}  \int_{d_g(z_1,z_2)<a_n, d_g(z_1,z_3)<5\log^2 n/\sqrt n, d_g(z_3,z_4)<a_n} \rho_{4} \left(z_1, z_{2}, z_{3}, z_4\right) \\&\quad \frac{\omega(z_2)}{\pi} \wedge  \frac{\omega(z_3)}{\pi} \wedge  \frac{\omega(z_4)}{\pi} \\ 
				\leq  
				&C\int_M \omega(z_1) \int_{|z_2|\leq a_n, |z_3|\leq 5\log^2 n/\sqrt n, |z_4-z_3|<a_n } \rho_{4} \left(0, z_{2}, z_{3}, z_4\right)  {\omega(z_2)} \wedge  {\omega(z_3)}\wedge  {\omega(z_4)}, 
			\end{split}
		\end{equation*} 
		which can be further  bounded by
		\beq \label{differered22} \begin{split}
			&Cn^{-3} \int_M \omega(z_1) \int_{|z_2|\leq n^{1/2}a_n, |z_3|\leq 5\log^2 n, |z_4-z_3|< n^{1/2}a_n}\rho_{4} \left(0, \frac{z_{2}}{\sqrt n},\frac{ z_{3}}{\sqrt n},\frac{ z_{4}}{\sqrt n} \right)\\&\quad {d} \ell_{z_{2}} \wedge {d}\ell_{z_{3}} \wedge {d}\ell_{z_{4}}  \\
			\leq &Cn^{}  \int_{|z_2|\leq n^{1/2}a_n, |z_3|\leq 5\log^2 n, |z_4-z_3|< n^{1/2}a_n} |z_2|^2 |z_4-z_3|^2 {d} \ell_{z_{2}}\wedge  {d}\ell_{z_{3}} \wedge {d}\ell_{z_{4}}  \\
			=&O(\log ^4n/n).  
		\end{split}\eeq	Therefore,  the cardinalities of the sets $\tilde{\mathcal I}_{n}$ and ${\mathcal I}_{n}$   are asymptotically the same in probability, which completes the proof of Lemma \ref{nssn}. 
	\end{proof}

	\subsection{Proof of Theorem \ref{main1}}\label{sec:pfmain1}
	As we explained before, to prove Theorem \ref{main1}, we only need to prove \eqref{tilderhok}. We define the sets  $$\label{omega1} \Omega_k: =\Big\{(z_1,..,z_k)\in U^k, \,\,\min_{i\neq j}d_g(z_i, z_j) \geq 4\log^2 n/\sqrt n+2a_n \Big\},$$
	and 
	$$\label{omega2} \tilde{\Omega}_k: =\Big\{(z_1,..,z_k)\in U^k, \,\,\min_{i\neq j}d_g(z_i, z_j) \geq 4\log^2 n/\sqrt n \Big\}.$$
	We  define the complementary set 
	$$\label{comset}   (\tilde{\Omega}_k)^c :=U^k\setminus  \tilde{\Omega}_k. $$
	
	Observe that, by the definition of $\mathlarger{\chi}_n$,
	it cannot contain two points with their distance smaller than $4\log^2 n/\sqrt n$. Consequently,
	$\tilde{\rho}_k(z_1,\ldots, z_k)=0$  if $d_g(z_i,z_j)<4\log^2 n/\sqrt n$ for some $i\neq j$. In other words, 
	\begin{equation}\label{0rho}
		\tilde{\rho}_k\equiv0 \,\,\mbox{on}\,\,(\tilde{\Omega}_k)^c.
	\end{equation} 
	On $\tilde{\Omega}_k$, by the definition of $\tilde {\mathcal I}_n$ in \eqref{tildein},  we have the  upper bound 
	\begin{equation}\label{trhoup}
		\begin{split}
			&\tilde{\rho}_k(z_1,\ldots, z_k) \\
			\leq & \left(\frac{1}{2}\right)^k \int_{ B_{a_n}(z_1)} \frac{\omega( w_{1})}{\pi} \cdots \int_{ B_{a_n}(z_k)} \frac{\omega( w_{k})}{\pi} \rho_{2k}\left(z_{1}, w_{1}, \ldots, z_{k}, w_{k}\right),
		\end{split}
	\end{equation}
	where  the factor of $(1/2)^k$ comes from the fact that the $z_i's$ here are chosen uniformly from each pair of zeros of  small distances.
	For the lower bound, we define two sets
	$$
	D_3(z_1,\ldots, z_k):=\cup_{i=1}^k B_{2a_n}(z_i)
	$$%
	and 
	\begin{align*}
		D_4(z_1,\ldots, z_k):=\Big\{(x,y):d_g(x,y)<a_n \mbox{ and }\exists\, 1\leq i\leq k, d_g(z_i,x)<5\log^2 n/\sqrt n\Big\}.
	\end{align*}
	We then set 
	\begin{equation*}
		\begin{split}
			E_1:= &\left(\frac{1}{2}\right)^k \int_{ B_{a_n}(z_1)}\frac{\omega( w_{1})}{\pi} \cdots\int_{ B_{a_n}(z_k)}\frac{\omega( w_{k})}{\pi} \int_{D_3(z_1,\ldots, z_k)} \frac{\omega( u_1)}{\pi} \\
			&\quad \rho_{2 k+1}\left(z_{1}, w_{1}, \ldots, z_{k}, w_{k},u_1\right),
		\end{split}
	\end{equation*}
	and
	\begin{equation}\label{defe2}
		\begin{split}
			E_2:=&	 \left(\frac{1}{2}\right)^k \int_{ B_{a_n}(z_1)}\frac{\omega( w_{1})}{\pi} \cdots\int_{ B_{a_n}(z_k)}\frac{\omega( w_{k})}{\pi} \int_{D_4(z_1,\ldots, z_k)} \frac{\omega( u_1)}{\pi}  \wedge \frac{\omega( u_2)}{\pi} \\
			&\quad \rho_{2 k+2}\left(z_{1}, w_{1}, \ldots, z_{k}, w_{k},u_1,u_2\right).
		\end{split}
	\end{equation}
	On $\Omega_k$, we have the lower bound  (see \cite[Lemma 2]{FGY} for an analogous
estimate in the setting of zeros of real Gausisan processes)
	\begin{equation}\label{lwoeredg}
		\begin{split}
			\tilde{\rho}_k(z_1,\ldots, z_k) \geq &
			\left(\frac{1}{2}\right)^k \int_{ B_{a_n}(z_1)}\frac{\omega( w_{1})}{\pi} \cdots\int_{ B_{a_n}(z_k)}\frac{\omega( w_{k})}{\pi} \rho_{2 k}\left(z_{1}, w_{1}, \ldots, z_{k}, w_{k}\right)\\
			&-E_1-E_2.
		\end{split}
	\end{equation}
	For the correlation functions of the auxiliary point process $\larger\chi_n$, we have 
	\begin{lem}\label{bigdistance}As $n\to\infty$, 
		$\tilde{\rho}_k$ is uniformly bounded on $\tilde{\Omega}_k$. Moreover,
		one has the uniform convergence, 
		\beq\label{spera}
		\tilde{\rho}_{k}^{}\left(z_{1}, \ldots, z_{k}\right) \to \left(\frac{a^4}8\right)^{k}\quad \mbox{in }\,\,\, \Omega_k. \eeq
		
		
	\end{lem} 
	Assume that Lemma \ref{bigdistance} is correct. 
	Then by   Lemma \ref{bigdistance}, identity \eqref{0rho}, and the fact that  the set $\tilde{\Omega}_k \setminus \Omega_k$ has a vanishing measure as $n\to\infty$,  we have:
	\begin{align*}&\lim_{n\to\infty}\int_{U^k}\tilde\rho_k\left({z_1},..,{ z_k}\right)\frac{\omega(z_1)}\pi\wedge \cdots \wedge \frac{\omega(z_k)}\pi\\ =&\lim_{n\to\infty}\int_{ \Omega_k  } \tilde\rho_k\left({z_1},..,{ z_k}\right)\frac{\omega(z_1)}\pi\wedge \cdots \wedge \frac{\omega(z_k)}\pi, 
	\end{align*}
	which is equal to (by the dominated convergence theorem)
	\begin{align*}
		& \int_{U^k}\lim_{n\to\infty}[ 1_{\Omega_k} \tilde{\rho}_{k}^{}\left(z_{1}, \ldots, z_{k}\right)] \frac{\omega(z_1)}\pi\wedge \cdots \wedge \frac{\omega(z_k)}\pi
		= \left(\frac{a^4}8 \int_U\frac{\omega}\pi\right)^{k}.\end{align*}
	This gives \eqref{tilderhok}, and thus we complete the proof of Theorem \ref{main1}. 
	All the remaining effort is to prove  Lemma \ref{bigdistance}. 
	\begin{proof}[Proof of Lemma \ref{bigdistance}] By \eqref{trhoup} and \eqref{lwoeredg}, the lemma will be established if we can prove that the integration involving $\rho_{2k}$ yields the   limit \eqref{spera}, and both $E_1$ and $E_2$ converge to 0. We examine these three terms individually.
		

\emph{Step 1.}		We first consider the main term	$$
		\left(\frac12\right)^k\int_{ B_{a_n}(z_1)}\frac{\omega( w_{1})}\pi \cdots\int_{ B_{a_n}(z_k)}\frac{\omega( w_{k})}\pi    \rho_{2 k}\left(z_{1}, w_{1}, \ldots, z_{k}, w_{k}\right).  $$
		
		Given $(z_1, \dots, z_k) \in \tilde{\Omega}_k$, we have $$\eta := \min_{i\neq j} d_g(z_{i},  z_j)\geq 4\log^2 n/\sqrt n.$$
		For the points $w_i's$ where $w_i \in B_{a_n} (z_i)$, we can observe $k$ clusters of pairs of points, forming the partition $\mathcal{P} = \{\{z_1, w_1\}, \dots, \{z_k, w_k\}\}$, and the distance between any two  clusters is greater than $\eta/2\geq 2\log^2 n/\sqrt{n}$. Thus, using \eqref{mainsplit2}, we obtain
		\beq\label{split1122}\rho_{2k}(z_1, w_1, .., z_k, w_k)=\rho_2(z_1, w_1)\cdots\rho_2(z_k, w_k)+O(n^{-\infty}).\eeq
		Therefore, by the limit \eqref{maintwolimit} in Lemma \ref{lunis}, 
		we further obtain, \begin{align*}&\lim_{n\to\infty}\int_{ \prod_{i=1}^k B_{a_n}(z_i) }\left(\frac12\right)^k\rho_{2k}(z_1, w_1,.., z_k, w_k) \frac{\omega( w_{1})}\pi\wedge \cdots \wedge\frac{\omega( w_{k})}\pi\\=&\lim_{n\to\infty}\prod_{i=1}^k\int_{   B_{a_n}(z_i) }\frac 12 \rho_2(z_i, w_i)\frac{\omega( w_{i})}\pi =\Big(\frac { a^4}8\Big)^k.\end{align*}
		This gives the upper bound for $\tilde{\rho}_k$ on $\tilde{\Omega}_k$, which is  also valid  on $\Omega_k$ since $\Omega_k \subset \tilde{\Omega}_k$.
		The rest is to prove a matching  lower bound for $\tilde{\rho}_k$ on $\Omega_k$ by \eqref{lwoeredg}.	

        \emph{Step 2.}
By discarding the constant $(1/2)^k$, we consider the error term $E_1$, 
		\begin{align*}\int_{ B_{a_n}(z_1)}\frac{\omega( w_{1})}{\pi} \cdots\int_{ B_{a_n}(z_k)}\frac{\omega( w_{k})}{\pi} \int_{D_3(z_1,\ldots, z_k)} \frac{\omega( u_1)}{\pi} \rho_{2 k+1}\big(z_{1}, w_{1}, \ldots, z_{k}, w_{k},u_1\big).
		\end{align*}	 For $n$ large enough, we first have $B_{2a_n}(z_i)  \cap  B_{2a_n}(z_j)  =\emptyset$ for distinct $z_i\neq z_j$ since $d_g(z_i, z_j)\geq 4\log^2 n/\sqrt n$. And thus $E_1$ can be split into the summation 
		$$
		\sum _{\ell=1}^k\int_{ B_{a_n}(z_1)}\frac{\omega( w_{1})}{\pi} \cdots\int_{ B_{a_n}(z_k)}\frac{\omega( w_{k})}{\pi} \int_{B_{2a_n}(z_\ell) } \frac{\omega( u_1)}{\pi} \rho_{2 k+1}\left(z_{1}, w_{1}, \ldots, z_{k}, w_{k},u_1\right).
		$$
		Assume $u_1\in B_{2a_n}(z_1)$,  i.e., both $w_1$ and $u_1$ fall in the geodesic ball $B_{2a_n}(z_1)$. In this case,  $z_1, w_1$ and $u_1$ are considered to be in the same cluster, and each pair of $z_i$ and $w_i$ are in the same cluster for $i=2,.., k$, i.e., we have the partition $\mathcal{P} = \{\{z_1, w_1, u_1\}, \dots, \{z_k, w_k\}\}$, and the distance between any two  clusters is greater than $\eta/2\geq 2\log^2 n/\sqrt{n}$.  By \eqref{mainsplit2}, as $n$ large enough, we have  the estimate,  
		$$\rho_{2k+1}(z_1, .., z_k, w_1,.., w_k, u_1)=\rho_3(z_1, w_1, u_1)\rho_2(z_2, w_2)\cdots \rho_2(z_k, w_k)+O(n^{-\infty}).$$
		By  the limit \eqref{maintwolimit},  we get 
		$$
		\begin{aligned}
			&\int_{ B_{a_n}(z_1)}\frac{\omega( w_{1})}{\pi} \cdots\int_{ B_{a_n}(z_k)}\frac{\omega( w_{k})}{\pi} \int_{B_{2a_n}(z_1) } \frac{\omega( u_1)}{\pi} \rho_{2 k+1}\left(z_{1}, w_{1}, \ldots, z_{k}, w_{k},u_1\right) \\ \leq  &C\Big[ \int_{B_{2a_n}(z_1)^2 }\rho_3(z_1, w_1, u_1)\omega(w_1)\wedge\omega(u_1) \Big]\Big[\prod_{i=2}^k   \int_{B_{a_n}(z_i) }\rho_2(z_i, w_i)\omega(w_i)\Big]\\
			\leq & C \int_{B_{2a_n}(z_1)^2 }\rho_3(z_1, w_1, u_1)\omega(w_1)\wedge \omega(u_1) =O(n^{-3/2}),
		\end{aligned}
		$$ where the last estimate follows the same lines of reasoning as in \eqref{differered}.  There are only $k$ terms in the summation, and thus $E_1$ decays with an order of $O(n^{-3/2})$, which is negligible in the limit. 

        \emph{Step 3.}		Finally we consider the error term $E_2$ in \eqref{defe2}.
		Since  $d_g(z_i, z_j)\geq 4\log^2 n/\sqrt n+2a_n$ for all $i\neq j$, we have 
		$$
		\cup_{i=1}^k \left(\cap_{j\neq i}[B_{ 2\log^2 n/\sqrt n+a_n}(z_{j})]^c\right)=M.
		$$
		For any $x\in M$, there must exist an $i$ such that $x\in \cap_{j\neq i}[B_{2\log^2 n/\sqrt n+a_n }(z_{j})]^c$. 
		Consequently, for $n$ large, we have $D_4(z_1,\ldots, z_k)=\cup_{i=1}^k D_{4,i}$ where
		$$
		D_{4,i}=D_4(z_1,\ldots, z_k) \cap \{(x,y)|x,y\in  \cap_{j\neq i}[B_{2\log^2 n/\sqrt n}(z_{j})]^c \}. 
		$$
		Here, we used the fact that, if $x \in \cap_{j\neq i}[B_{2\log^2 n/\sqrt n+a_n}(z_{j})]^c$, then for  $n$ large enough, $y \in \cap_{j\neq i}[B_{\log^2 n/\sqrt n}(z_{j})]^c$ provided that $d_g(x,y)<a_n$. 
		
		Then $E_2$ can be bounded above by   
		\begin{align*} &\sum_{\ell=1}^k\int_{ B_{a_n}(z_1)}\frac{\omega( w_{1})}{\pi} \cdots\int_{ B_{a_n}(z_k)}\frac{\omega( w_{k})}{\pi} \int_{D_{4,\ell} } \frac{\omega( u_1)}{\pi}  \wedge \frac{\omega( u_2)}{\pi}\\& \quad \rho_{2 k+2}\left(z_{1}, w_{1}, \ldots, z_{k}, w_{k},u_1,u_2\right).
		\end{align*}
		Without loss of generality,	assume $(u_1, u_2)\in D_{4,1}$. In this case, the two pairs $(z_1, w_1)$, $(u_1,u_2)$  are considered to be in the same cluster, and  $(z_i, w_i)$ are in the same cluster for $i=2,.., k$, i.e., we have the partition $\mathcal{P}$ $=$ $\{\{z_1, w_1, u_1, u_2\}$, $\{z_2, w_2\}$, $\dots$, $\{z_k, w_k\}\}$, and the distance between any two  clusters is greater than $\log^2 n/\sqrt{n}$.  
		By \eqref{mainsplit2} again, as $n$ large enough,  we have the asymptotic expansion,  
		\begin{align*}\rho_{2k+2}(z_1,w_1 .., z_k, w_k, u_1, u_2)=\rho_4(z_1, w_1, u_1, u_2)\rho_2(z_2, w_2)\cdots \rho_2(z_k, w_k)+O(n^{-\infty}).\end{align*} 
		As we argued before, the integration of  $\prod_{i=2}^k\rho_2(z_i, w_i)$ with respect to variables $w_2, .., w_k$ is bounded above by some constant. Therefore, the integration of $\rho_{2k+2}$  has the upper bound, 
		\begin{align}\label{2k+2bd1}
			C\int_{d_g(z_1, w_1)<a_n, (u_1,u_2)\in D_{4,1}}  \quad\rho_4(z_1, w_1, u_1, u_2) \omega(w_1)\wedge\omega(u_1)\wedge\omega(u_2).
		\end{align}
		We now define
		$$
		D_{4,1,1}=D_{4,1}\cap \{(x,y)|d_g(x,z_1)<2\log^2 n/\sqrt n\}.
		$$
		And, for $2\leq j\leq k$, let
		$$
		D_{4,1,j}=D_{4,1}\cap \{(x,y)|d_g(x,z_1)\geq 2\log^2 n/\sqrt n, d_g(x,z_j)<5\log^2 n/\sqrt n\}.
		$$
		Then by the definition of $D_4$, we have $D_{4,1}=\cup_{j=1}^k D_{4,1,j}$ and the integration in \eqref{2k+2bd1} can be bounded above by
		$$
		C\sum_{j=1}^k 	\int_{d_g(z_1, w_1)<a_n, (u_1,u_2)\in D_{4,1,j}}  \quad\rho_4(z_1, w_1, u_1, u_2) \omega(w_1)\wedge\omega(u_1)\wedge\omega(u_2). 
		$$  
		By the definition of the correlation function, the term for $j=1$  can be bounded above by the expected number of tuples of 
		$(w_1, u_1, u_2)$ satisfying the conditions $d_g(z_1,w_1)<a_n$, $d_g(u_1,u_2)<a_n$, and $d_g(u_1,z_1)<2\log^4 n/\sqrt n$, which is of the order $O(\log^4 n/n)$, as derived in \eqref{differered22}.
		
		By the definition of $D_{4,1,j}$ for $j\geq 2$, 
		$d_g(\{z_1,w_1\},\{u_1,u_2\})\geq \log^2 n/\sqrt{n}$. And thus
		we have the factorization 
		$$
		\rho_4(z_1,w_1,u_1,u_2)=\rho_2(z_1,w_1)\rho_2(u_1,u_2)+
		O(n^{-\infty}).
		$$ 
		The integration of $\rho_2(z_1,w_1)$ with respect to $w_1$ gives rise to a bounded factor as before. To estimate the integration of  $\rho_2(u_1,u_2)$, we choose a normal coordinate around $z_j$ to get
		\begin{align*}
			&\int_{(u_1,u_2)\in D_{4,1,j}} \rho_2(u_1,u_2)\omega(u_1)\wedge \omega(u_2)\\
			\leq &C
			\int_{  |u_1-u_2|<n^{1/2}a_n,  |u_1|<5\log^2 n  }
			\quad\rho_2(\frac{u_1}{\sqrt n}, \frac{u_2}{\sqrt n}) 
			\omega(\frac{u_1}{\sqrt n})
			\wedge\omega(\frac{u_2}{\sqrt n})
			\\ 
			\leq &C\int_{  |u_1- u_2|<n^{1/2}a_n,  |u_1|<5\log^2 n}   |u_1-u_2|^2  d\ell_{u_1}\wedge d\ell_{u_2}=O(\log^4 n/n).
		\end{align*} 
		Since $D_4=\cup_{1\leq i,j\leq k}D_{4,i,j}$ which consists of $k^2$ sets, and the integration on each $D_{4,i,j}$ decays with order $O(\log^4 n/n)$, we can conclude that the third term on the right hand side of \eqref{lwoeredg} also tends to 0 in the limit, and thus we complete the proof of Lemma \ref{bigdistance}.
	\end{proof}

\section*{Acknowledgements}
We thank the anonymous referees for carefully reading the manuscript and providing valuable comments, which have helped to improve the clarity of the paper. Dong Yao is supported by the National Key R\&D Program of China (No.\ 2023YFA1010101) and the NSFC grant (No.  12571161).

\end{document}